\documentclass[sn-mathphys]{sn-jnl}

\usepackage{natbib}

\newcommand{\bbHr}{\mathbb{H}^2(r)}
\newcommand{\bfe}{\mathbf{e}}
\newcommand{\rmd}{\mathrm{d}}
\newcommand{\rme}{\mathrm{e}}

\jyear{2021}%

\theoremstyle{thmstyleone}%
\newtheorem{maintheorem}{Main Theorem}
\newtheorem{theorem}{Theorem}
\newtheorem{proposition}[theorem]{Proposition}%
\newtheorem{lemma}[theorem]{Lemma}

\theoremstyle{thmstyletwo}%
\newtheorem{remark}{Remark}%

\theoremstyle{thmstylethree}%
\newtheorem{definition}{Definition}%

\begin{document}

\title{Catenaries and minimal surfaces of revolution in hyperbolic space}

\author[1]{\fnm{Luiz C. B. da Silva}}\email{luiz.da-silva@weizmann.ac.il}

\author[2]{\fnm{Rafael L\'opez}}\email{rcamino@ugr.es}
\equalcont{These authors contributed equally to this work.}

\affil[1]{\orgdiv{Department of Physics of Complex Systems, Weizmann Institute of Science. 
 Rehovot 7610001, Israel}}

\affil[2]{\orgdiv{Departamento de Geometr\'ia y Topolog\'ia, Universidad de Granada. 18071 Granada, Spain}}

\abstract{We introduce the concept of extrinsic catenary in the hyperbolic plane. Working in the hyperboloid model, we define an extrinsic catenary as the shape of a curve hanging under its weight as seen from the ambient space. In other words, an extrinsic catenary is a critical point of the potential functional, where we calculate the potential with the extrinsic distance to a fixed reference plane in the ambient Lorentzian space. We then characterize extrinsic catenaries in terms of their curvature and as a solution to a prescribed curvature problem involving certain vector fields. In addition, we prove that the generating curve of any minimal surface of revolution in the hyperbolic space is an extrinsic catenary with respect to an appropriate reference plane. Finally, we prove that one of the families of extrinsic catenaries admits an intrinsic characterization if we replace the extrinsic distance with the intrinsic length of horocycles orthogonal to a reference geodesic.}
 



\maketitle

\section{Introduction}\label{sec1}

The catenary is the solution to the problem of minimizing the potential gravitational energy of a chain hanging under its weight when supported only at its ends. Euler proved that by rotating a catenary about the reference line with respect to which the weight is measured, the resulting surface of revolution has zero mean curvature \cite{eulercatenoid,struik}. This surface is called a catenoid, the only non-planar minimal surface of revolution in Euclidean space. Recently, the second author extended the notion to the catenary in the sphere and the hyperbolic plane \cite{lopezcatenaryform}. After we fix a reference geodesic $\ell$, the energy functional to be minimized is the potential energy, which is given by integrating the distance of points of the chain to $\ell$. A critical point of this functional is called an (intrinsic) catenary.
 

Embedding the hyperbolic plane (2-sphere) in the hyperbolic space (3-sphere, respectively), we may ask, as Euler did in Euclidean space, whether the revolution of a catenary about $\ell$ is a minimal surface. It turns out that the resulting surface of revolution is not minimal. This problem was circumvented in the spherical case by introducing the notion of an extrinsic catenary. More precisely, instead of measuring the potential using the intrinsic distance, one uses the distance in the ambient Euclidean space to a plane passing by the center of the sphere \cite{lopezcatenaryform}. This paper aims to provide a similar extension of Euler's result to the hyperbolic space.
 
Contrarily to the three-dimensional Euclidean space and sphere, in the hyperbolic space, $\mathbb{H}^3(r)$, three types of surfaces of revolution exist. Indeed, considering $\mathbb{H}^3(r)$ as a hypersurface of the $4$-dimensional Lorentz-Minkowski space  $\mathbb{E}_1^4$ (the hyperboloid model), rotations in $\mathbb{H}^3(r)$ correspond to orthogonal transformations in $\mathbb{E}_1^4$ that leave a 2-dimensional subspace $P^2$ point-wise fixed. (We refer to $P^2$ as the axis of revolution.) Thus, depending on the causal character of the axis of revolution $P^2\subset\mathbb{E}_1^4$, the corresponding surface of revolution is of elliptic, hyperbolic, or parabolic type\footnote{A surface of revolution with axis $P^2$ is of elliptic type if $P^2$ is spacelike, hyperbolic type if $P^2$ is timelike, and parabolic type if $P^2$ is lightlike. The terminology stands for the fact that the orbits of the revolution are ellipses, hyperbolas, and parabolas, respectively.}. Motivated by this classification, instead of measuring the gravitational potential in the hyperbolic plane $\mathbb{H}^2(r)$ from the intrinsic distance to a fixed geodesic of $\mathbb{H}^2(r)$, as in Ref. \cite{lopezcatenaryform}, the potential will be measured using the extrinsic distance. More precisely, the {\it extrinsic catenary problem} consists in finding the shape of a curve $\gamma\colon [a,b]\to \mathbb{H}^2(r)\subset\mathbb{H}^3(r)$, which is a critical point of the potential energy 
$$
 \gamma\longmapsto \int_a^b \mbox{dist}_{\mathbb{E}_1^4}(\gamma(t),P^2)\Vert\dot{\gamma}(t)\Vert\,\rmd t,
$$
where $\mbox{dist}_{\mathbb{E}_1^4}(\gamma(t),P^2)$ is the distance in $\mathbb{E}_1^4$ between $P^2$ and the point $\gamma(t)$. 

Depending on the causal character of $P^2$, critical points of the potential energy functional will be called extrinsic catenaries of elliptic, hyperbolic, and parabolic types. Our main theorem will extend Euler's result about the catenoid to the hyperbolic space. Namely, we have (we prove it as Theorem \ref{tmain} in Section \ref{sec4})


\begin{maintheorem}
The generating curves of minimal surfaces of revolution in $\mathbb{H}^3(r)$ of elliptic, hyperbolic, and parabolic types are extrinsic catenaries of spherical, hyperbolic, and parabolic types, respectively.
\end{maintheorem}

In addition, we show that one of the families of extrinsic catenaries admits an intrinsic characterization in terms of the so-called horo-catenary. In other words, a horo-catenary is a critical point of the potential energy functional obtained by replacing $\mbox{dist}_{\mathbb{E}_1^4}(\gamma(t),P^2$ with the length of a horocycle orthogonal to a fixed reference geodesic. Thus, we obtain the following theorem (we prove it as Theorem \ref{thr::EllipticCatenaryAreHorocatenary} in Section \ref{sect::Horocatenary})

\begin{maintheorem}
    Every hyperbolic horo-catenary is an extrinsic catenary in $\mathbb{H}^2(r)$ of the elliptic type. Consequently, the generating curves of minimal surfaces of revolution in $\mathbb{H}^3(r)$ of the elliptic type are horo-catenaries.
\end{maintheorem}

Surfaces of revolution in $\mathbb{H}^3(r)$ are well-known objects of investigation. First, Mori found minimal surfaces of revolution of elliptic type \cite{morirotation}.\footnote{Surfaces of revolution of elliptic type are often called surfaces of revolution of spherical type. However, unless $P^2$ is the $yz$-plane (the $x$-axis is timelike), the orbits are not (Euclidean) circles.} Later, do Carmo and Dajczer obtained all the minimal surfaces of revolution and, among other results,  established their relation with helicoids \cite{carmodajczerrotation}. To our knowledge, this work is the first to characterize minimal surfaces of revolution in hyperbolic space through a variational formulation for their generating curves. 
 
The rest of this paper is divided as follows. In Section \ref{sec2}, we formulate the extrinsic catenary problem in $\mathbb{H}^2(r)$. Theorem \ref{t4} characterizes extrinsic catenaries in terms of their curvature function. In Section \ref{sec22}, we characterize extrinsic catenaries by prescribing the curvature in terms of an equation involving Killing vector fields of $\mathbb{E}_1^3$ (Theorem \ref{tkilling}). These vector fields indicate the direction of the geodesics in $\mathbb{E}_1^3$ used to measure the weight. To be self-contained, in Section \ref{sec3}, we provide a detailed construction of all surfaces of revolution in hyperbolic space along with the computation of their mean curvature.
In Section \ref{sec4}, we prove our main result concerning the characterization of the generating curves of surfaces of revolution in the hyperbolic space (Theorem \ref{tmain}). In Section \ref{sect::Horocatenary}, we prove that horocatenaries are extrinsic catenaries of the elliptic type, thus providing an intrinsic characterization for minimal surfaces of revolution of the elliptic type (Theorem \ref{thr::EllipticCatenaryAreHorocatenary}). Finally, we present our concluding
remarks and formulate some open problems in the last section.

\section{The extrinsic catenary problem}
\label{sec2}

Let $\mathbb{E}_1^{n+1}=(\mathbb{R}^{n+1},\langle X,Y\rangle_1=-X_0Y_0+\sum_{i=1}^nX_iY_i)$ be the $(n+1)$-dimensional Lorentz-Minkowski space. In the hyperboloid model of the hyperbolic space, we see $\mathbb{H}^n(r)$ as a hypersurface in $\mathbb{E}_1^{n+1}$ of curvature $K=-1/r$: 
\begin{equation}
    \mathbb{H}^n(r) = \{X=(X_0,\dots,X_n)\in\mathbb{E}_1^{n+1}:\langle X,X\rangle_1=-r^2,\,X_0>0\}.
\end{equation}
We shall denote by $\langle\cdot,\cdot\rangle$ the induced inner product of $\mathbb{H}^n(r)$. In addition, we embed $\mathbb{H}^2(r)$ into $\mathbb{H}^3(r)$ by the natural inclusion 
\begin{equation}\label{inclusion}
(x,y,z)\in \mathbb{H}^2(r)\hookrightarrow (x,y,z,0)\in \mathbb{H}^3(r).
\end{equation}
The extrinsic catenary problem is formulated as follows:

\begin{quote}{\it The extrinsic catenary problem}. Given a subspace $P^2$ of $\mathbb{E}_1^{4}$, find the shape of a chain $\gamma:[a,b]\to \mathbb{H}^2(r)\subset\mathbb{H}^3(r)$ that optimizes the potential energy
\begin{equation}\label{exc1}
\mathcal{W}[\gamma]=\int_a^b \mbox{dist}(\gamma(t),P^2)\Vert\dot{\gamma}(t)\Vert\,\rmd t,
\end{equation}
where $\mbox{dist}(\gamma(t),P^2)$ is the distance in $\mathbb{E}_1^{4}$ of the point $\gamma(t)$ to $P^2$. 
\end{quote} 

\begin{definition} Critical points of $\mathcal{W}[\gamma]$  are called \emph{extrinsic catenaries}. \end{definition}

In the definition of extrinsic catenaries, it is implicitly assumed that $\gamma$ lies in one side of the plane $P^2$, which then implies that the integrand in \eqref{exc1} is positive. In contrast to how catenaries are defined in Ref. \cite{lopezcatenaryform}, the weight is now calculated using the extrinsic distance in $\mathbb{E}_1^4$, rather than the intrinsic distance  to a given geodesic in $\mathbb{H}^2(r)$.
 
In $\mathbb{H}^3(r)$, there are three types of planes $P^2$ according to their causal character. In general, if $P^k$ denotes a $k$-dimensional subspace of $\mathbb{E}_1^{4}$, $P^k$ is said to be \emph{Lorentzian}, \emph{Riemannian}, or \emph{degenerate}, if the restriction of the metric $\langle\cdot,\cdot\rangle_1$ to $P^k$ is a Lorentzian, Riemannian, or degenerate metric, respectively. Therefore, we define three types of extrinsic catenaries: $\gamma$ is an extrinsic catenary of (i) \emph{elliptic type}, (ii) \emph{hyperbolic type}, and (iii) \emph{parabolic type} if $\gamma$ is a critical point of \eqref{exc1} when $P^2$ is   (i) Lorentzian, (ii) Riemannian, and (iii) degenerate, respectively.
 
To characterize the extrinsic catenaries, we will obtain the Euler-Lagrange equations of the functional \eqref{exc1}. This task can be significantly simplified by choosing a suitable system of coordinates and then applying the standard techniques of the Calculus of Variations. 

Without loss of generality, let $\ell$ be the geodesic in $\bbHr$ obtained from the intersection with the plane of equation $z=0$:
\begin{equation}
 \ell(v) = r(\cosh\frac{v}{r},\sinh\frac{v}{r},0), \quad v\in\mathbb{R}.
\end{equation}
If $\beta_v(u)$ denotes the geodesic with  unit velocity $X(v)=(0,0,1)\in T_{\ell(v)}\mathbb{H}^2(r)$, $\langle X(v),\ell'(v)\rangle_1=0$, then we can parametrize $\mathbb{H}^2(r)$ by 
\begin{eqnarray}\label{eq35}
\psi(u,v) & = & \beta_v(u) = \ell(v)\cosh \frac{u}{r}+\sinh \frac{u}{r}\,X\nonumber\\
 & = & r(\cosh\frac{u}{r}\cosh\frac{v}{r},\cosh\frac{u}{r}\sinh\frac{v}{r},\sinh\frac{u}{r}).
\end{eqnarray} 
This coordinate system is known as the semi-geodesic coordinates \cite{struik}. Since 
\begin{equation}\label{puv}
\begin{split}
 \psi_u &= (\sinh\frac{u}{r}\cosh\frac{v}{r},\sinh\frac{u}{r}\sinh\frac{v}{r},\cosh\frac{u}{r})\\
 \psi_v& = (\cosh\frac{u}{r}\sinh\frac{v}{r},\cosh\frac{u}{r}\cosh\frac{v}{r},0),
\end{split}
\end{equation}
the induced metric on $\mathbb{H}^2(r)$ takes the form
\begin{equation}\label{metric}
 \rmd s^2 = \rmd u^2+\cosh^2\frac{u}{r}\,\rmd v^2.
\end{equation}

Now, we compute the distance in $\mathbb{E}_1^4$ from $\gamma(t)$ to $P^2$. Let $\{\mathbf{e}_x,\mathbf{e}_y,\mathbf{e}_z,\mathbf{e}_w\}$ be the unit velocity vectors of the canonical Cartesian coordinates $(x,y,z,w)$ of $\mathbb{E}_1^4$ and denote $[\mathbf{v}_1,\dots,\mathbf{v}_k]=\mbox{span}\{\mathbf{v}_1,\dots,\mathbf{v}_k\}$. Without loss of generality, we can suppose that the subspace $P^2$ is one of the following: 
\begin{enumerate}
\item $P^2=[\bfe_{x},\bfe_{y}]$ (Lorentzian).
\item $P^2=[\mathbf{e}_y,\mathbf{e}_z]$ (Riemannian).
 \item $P^2= [\mathbf{e}_{y}+\mathbf{e}_{x},\mathbf{e}_z]$ (degenerate).
 \end{enumerate}
From \eqref{eq35}, we have 
\[
 \mbox{dist}(\gamma(u,v),[\bfe_{x},\bfe_{y}]) = r\sinh\frac{u}{r},
\]
\[
 \mbox{dist}(\gamma(u,v),[\mathbf{e}_y,\mathbf{e}_z]) = r\cosh\frac{u}{r}\cosh\frac{v}{r},
\]
and
\[
 \mbox{dist}(\gamma(u,v),[\mathbf{e}_{y}+\mathbf{e}_{x},\mathbf{e}_z]) = \frac{1}{\sqrt{2}}\rme^{-\frac{v}{r}}\cosh\frac{u}{r}.
\]

Up to multiplication by a constant, the distance to $P^2$ can be taken as $\sinh\frac{u}{r}$ (elliptic), $\cosh\frac{u}{r}\cosh\frac{v}{r}$ (hyperbolic), and $e^{-v/r}\cosh\frac{u}{r}$ (parabolic). Therefore,
\begin{enumerate}
\item Extrinsic catenaries of elliptic type are critical points of
\begin{equation}\label{fu1}
 \mathcal{W}_{\mathcal{E}}[u(t),v(t)] = \int_a^b\sinh\frac{u}{r}\sqrt{\dot{u}^2+\dot{v}^2\cosh^2\frac{u}{r}}\,\rmd t,\quad u>0.
\end{equation}
\item Extrinsic catenaries of hyperbolic type are critical points of 
\begin{equation}\label{fu2}
 \mathcal{W}_{\mathcal{H}} [u(t),v(t)]= \int_a^b\cosh\frac{u}{r}\cosh\frac{v}{r}\sqrt{\dot{u}^2+\dot{v}^2\cosh^2\frac{u}{r}}\,\rmd t,\quad u>0.
\end{equation}
\item Extrinsic catenaries of parabolic type are critical points of 
\begin{equation}\label{fu3}
 \mathcal{W}_{\mathcal{P}}[u(t),v(t)] = \int_a^b\rme^{-\frac{v}{r}}\cosh\frac{u}{r}\sqrt{\dot{u}^2+\dot{v}^2\cosh^2\frac{u}{r}}\,\rmd t,\quad u>0.
\end{equation}
The assumption $u>0$ means that $\gamma$ lies on one side of the plane $P^2$.
\end{enumerate}

The functionals $\mathcal{W}_{\mathcal{E}}$, $\mathcal{W}_{\mathcal{H}}$, and $\mathcal{W}_{\mathcal{P}}$ are all of the form $\int f(u,v) \Vert\dot{\gamma}\Vert\rmd t$. To compute the corresponding Euler-Lagrange equations, we first need the expression of the geodesic curvature $\kappa$ in $\bbHr$ with respect to coordinates \eqref{eq35}. 

\begin{lemma}\label{lemma::CurvatureSemiGeoCoord}
If $\gamma(t)=\psi(u(t),v(t))$ is a smooth curve in $\mathbb{H}^2(r)$, its signed curvature in $\bbHr$ is given by
\begin{equation}\label{kk}
 \kappa = -\frac{1}{\Vert\dot{\gamma}(t)\Vert^3}\left(  \frac{2\dot{u}^2}{r}\sinh \frac{u}{r}+\frac{\dot{v}^2}{r}\cosh^2\frac{u}{r}\sinh\frac{u}{r}\right).
\end{equation}
\end{lemma}
\begin{proof}
Let $\Gamma_{ij}^k$ be the Christoffel symbols of the metric $\rmd s^2$ associated with the coordinates \eqref{eq35}. The expression for $\kappa$ is then given by 
\[
 \kappa = \frac{\sqrt{g_{ij}}\left(\Gamma_{22}^1\dot{v}^3-\Gamma_{11}^2\dot{u}^3-(2\Gamma_{12}^2-\Gamma_{11}^1)\dot{u}^2\dot{v}+(2\Gamma_{12}^1-\Gamma_{22}^2)\dot{u}\dot{v}^2+\ddot{u}\dot{v}-\dot{u}\ddot{v}\right)}{(g_{11}\dot{u}^2+2g_{12}\dot{u}\dot{u}+g_{22}\dot{v}^2)^{3/2}}.
\]
For the metric $\rmd s^2$ given in \eqref{metric}, we have $g_{11}=1$, $g_{12}=0$, and $g_{22}=\cosh^2\frac{u}{r}$. As a consequence, $\Gamma_{11}^1=\Gamma_{12}^1=\Gamma_{11}^2=\Gamma_{22}^2=0$ and 
$$
  \Gamma_{22}^1=-\frac{1}{r}\sinh\frac{u}{r}\cosh\frac{u}{r},\quad \Gamma_{12}^2=\frac{1}{r}\tanh \frac{u}{r}.
$$
Now, Eq. \eqref{kk} is immediate.
\end{proof}

\begin{lemma}\label{lemma::GeneralEL-EqsScalarVelocity}
Consider the functional
\begin{equation}\label{eq::GeneralFunctionalSemi-GeodCoord}
 \mathcal{E}_f[u(t),v(t)] = \int_a^b f(u,v) \Vert\dot{\gamma}(t)\Vert\, \rmd t,
\end{equation}
where $f$ is some smooth, positive function. The Euler-Lagrange equations of $\mathcal{E}_f$ are
\begin{equation}\label{el0}
\begin{split}
 \dot{v}\left[f\cosh\frac{u}{r}\left(\kappa-\frac{\dot{v}f_u\cosh\frac{u}{r}}{f\Vert\dot{\gamma}\Vert}  \right)+\frac{\dot{u}f_v}{\Vert\dot{\gamma}\Vert}\right] &= 0\\
 \dot{u}\left[f\cosh\frac{u}{r}\left(\kappa-\frac{\dot{v}f_u\cosh\frac{u}{r}}{f\Vert\dot{\gamma}\Vert}\right)+\frac{\dot{u}f_v}{\Vert\dot{\gamma}\Vert}\right] &= 0,
 \end{split}
\end{equation}
where $\kappa$ is the geodesic curvature of $\gamma$ in $\mathbb{H}^2(r)$. If, in addition, $\gamma$ is regular, then $\gamma$ is a solution of the Euler-Lagrange equations if, and only if, 
\begin{equation}\label{k-c}
 \kappa=\frac{1}{f\Vert\dot{\gamma}\Vert}\left(\dot{v} f_u\cosh\frac{u}{r}-\frac{\dot{u}f_v}{\cosh\frac{u}{r}}\right). 
\end{equation}
\end{lemma}

\begin{proof} 
The Euler-Lagrange equations associated with the Lagrangian $\mathcal{L}=f\Vert\dot{\gamma}\Vert$ are
\begin{equation}\label{el1}
\frac{\partial \mathcal{L}}{\partial u}-\frac{\rmd}{\rmd t}\frac{\partial \mathcal{L}}{\partial \dot{u}} = 0 \quad \mbox{and} \quad \frac{\partial \mathcal{L}}{\partial v}-\frac{\rmd}{\rmd t}\frac{\partial \mathcal{L}}{\partial \dot{v}} = 0.
\end{equation}
These equations are
\begin{align*}
f_u\Vert\dot{\gamma}\Vert+\frac{f\cosh\frac{u}{r}\sinh\frac{u}{r}\dot{v}^2}{r\Vert\dot{\gamma}\Vert}-\frac{\rmd}{\rmd t}\left(\frac{f\dot{u}}{\Vert\dot{\gamma}\Vert}\right)&=0\\
 f_v\Vert\dot{\gamma}\Vert-\frac{\rmd}{\rmd t}\left(\frac{f\dot{v}\cosh^2\frac{u}{r}}{\Vert\dot{\gamma}\Vert}\right)&=0.
 \end{align*}
 Equivalently, we have
 \begin{align*}
 \frac{\dot{u}^2f_v-\dot{u}\dot{v}f_u\cosh^2 \frac{u}{r}-\frac{2}{r}\dot{u}\dot{v}f\sinh\frac{u}{r}\cosh\frac{u}{r}}{\Vert\dot{\gamma}\Vert}-
 f\cosh^2\tfrac{u}{r}\frac{\rmd}{\rmd t}\left(\frac{\dot{v}}{\Vert\dot{\gamma}\Vert}\right)&=0\\
 \frac{\dot{v}^2f_u\cos^2 \frac{u}{r}-\dot{u}\dot{v}f_v-\frac{1}{r}\dot{v}^2f\sinh\frac{u}{r}\cosh\frac{u}{r}}{\Vert\dot{\gamma}\Vert}-
 f \frac{\rmd}{\rmd t}\left(\frac{\dot{u}}{\Vert\dot{\gamma}\Vert}\right)&=0.
 \end{align*}
After some manipulations, both equations can be expressed as 
\begin{eqnarray*}
0 & = & \dot{v}\Big[\frac{\dot{u}f_v-\dot{v}f_u\cosh^2\frac{u}{r} }{\Vert\dot{\gamma}\Vert} \\
&  & - \frac{f\cosh\frac{u}{r}}{\Vert\dot{\gamma}\Vert^3}
\left(\cosh\tfrac{u}{r}(\dot{u}\ddot{v}-\dot{v}\ddot{u})+\frac{\dot{v}\sinh\frac{u}{r}}{r}(2\dot{u}^2+\cosh^2\tfrac{u}{r}\dot{v}^2)\right)\Big]
\end{eqnarray*}
and
\begin{eqnarray*}
0 & = & \dot{u}\Big[\frac{\dot{u}f_v-f_u\dot{v}\cosh^2\frac{u}{r} }{\Vert\dot{\gamma}\Vert} \\
& & -\frac{f\cosh\frac{u}{r}}{\Vert\dot{\gamma}\Vert^3}
\left(\cosh\frac{u}{r}(\dot{u}\ddot{v}-\dot{v}\ddot{u})+\frac{\dot{v}\sinh\frac{u}{r}}{r}(2\dot{u}^2+\cosh^2\frac{u}{r}\dot{v}^2)\right)\Big].
\end{eqnarray*}
The identities \eqref{el0} are now obtained using the above two equations and the expression of $\kappa$ given in Eq. \eqref{kk}.
\end{proof}

Finally, we are in a position to characterize  the extrinsic catenaries in the hyperbolic plane.

\begin{theorem} \label{t4}
Let $\gamma(t)=\psi(u(t),v(t))$ be a regular curve in $\mathbb{H}^2(r)$. Then, $\gamma$ is an extrinsic catenary if, and only if, its curvature $\kappa$ in $\mathbb{H}^2(r)$ satisfies:
\begin{enumerate}
\item Elliptic type:
\begin{equation}
 \kappa = \frac{\dot{v}\cosh^2\frac{u}{r} }{r\sinh\frac{u}{r}\Vert\dot{\gamma}\Vert}.
\end{equation}
\item Hyperbolic type: 
\begin{equation}
 \kappa = \frac{1}{r \Vert\dot{\gamma}\Vert}\left( - \frac{\dot{u}\sinh\frac{v}{r}}{ \cosh\frac{u}{r}\cosh\frac{v}{r}} +\dot{v} \sinh\frac{u}{r}\right) .
\end{equation}
\item Parabolic type: 
\begin{equation}
 \kappa =\frac{1}{r\Vert\dot{\gamma}\Vert}\left(\frac{\dot{u}}{\cosh\frac{u}{r}} + \dot{v}\sinh\frac{u}{r}\right).
\end{equation}

 \end{enumerate}
\end{theorem}

\begin{proof} 
Just apply Lemma \ref{lemma::GeneralEL-EqsScalarVelocity} to the functionals in Eqs. \eqref{fu1}, \eqref{fu2}, and \eqref{fu3}. In other words,  choose $f$ in Eq. \eqref{k-c} as $\sinh\frac{u}{r}$, $\cosh\frac{u}{r}\cosh\frac{v}{r}$, and $\rme^{-\frac{v}{r}}\cosh\frac{u}{r}$, respectively.
\end{proof}
 
\begin{remark} 
In the variational formulation of the extrinsic catenary problem, Section \ref{sec1}, we implicitly assume that the chain's length is prescribed. Thus, there is a constraint $\int_a^b\Vert\dot{\gamma}(t)\Vert\, \rmd t=c$.   Consequently, the quantity $\mbox{dist}(\gamma(t),P^2)$ in \eqref{exc1}  should be replaced by $\mbox{dist}(\gamma(t),P^2)+\lambda$, where $\lambda$ is a Lagrange multiplier. In the functional \eqref{eq::GeneralFunctionalSemi-GeodCoord}, adding a Lagrange multiplier implies that the function $f$ should be replaced by  $f+\lambda$. The corresponding Euler-Lagrange equations coincide with those of Theorem \ref{t4} after a translation in $\mathbb{E}_1^4$ of the plane $P^2$.
\end{remark}

\section{Characterization of extrinsic catenaries}
\label{sec22}

In this section, we provide a coordinate-free characterization of extrinsic catenaries with the help of Killing vector fields of $\mathbb{E}_1^4$. The motivation comes from the Euclidean catenary 
$y(x)=\frac{1}{a}\cosh(ax+b)$ in the $(x,y)$-plane. This curve is the solution to the hanging chain problem in the Euclidean plane. In this setting, the reference line is the horizontal line $\ell$ of equation $y=0$, and the corresponding Euler-Lagrange equation is 
\begin{equation}\label{euclideo1}
 \frac{y''}{(1+y'^2)^{3/2}}=\frac{1}{y\sqrt{1+y'^2}}.
\end{equation}
The left-hand side is just the curvature $\kappa$ of the curve $y=y(x)$. The right-hand side is the Euclidean product ${\bf n}\cdot\mathbf{e}_y$, where $\mathbf{n}=(-y',1)/\sqrt{1+y'^2}$ is the unit normal of $y=y(x)$ and $-\mathbf{e}_y=(0,-1)$ is the direction of gravity in Euclidean plane. Thus, Eq. \eqref{euclideo1} can be expressed as 
\begin{equation}\label{euclideo2}
\kappa = \frac{{\bf n}\cdot\mathbf{e}_y}{\mbox{dist}(\gamma,\ell)}.
\end{equation}
Observe that the direction of gravity $-\mathbf{e}_y$ also has a geometric interpretation. It gives the direction of the geodesics orthogonal to $\ell$ used to compute $\mbox{dist}(\gamma,\ell)$.
 
We can provide a similar result for extrinsic catenaries in hyperbolic space. However, since the potential energy \eqref{exc1} measures the distance in $\mathbb{E}_1^4$ from the plane $P^2$, it is natural to expect that the right-hand side of Eq. \eqref{euclideo2} will incorporate the extrinsic nature of \eqref{exc1}. Such a characterization of hyperbolic extrinsic geodesics will be obtained as a corollary of the more general result

\begin{theorem}\label{t1}
A regular curve $\gamma(t)=\psi(u(t),v(t))$ in $\mathbb{H}^2(r)$ is a critical point of the functional \eqref{eq::GeneralFunctionalSemi-GeodCoord} if, and only if, its curvature $\kappa$ in $\mathbb{H}^2(r)$ satisfies

\begin{equation}\label{et1}
 \kappa = - \frac{\langle\mathbf{n}, \nabla f\rangle}{f},
 \end{equation}
 where $\mathbf{n}$ is the unit normal of $\gamma$ and $\nabla f$ is the gradient vector field of $f$. 
\end{theorem}
\begin{proof}
The tangent vector of $\gamma$ is $\dot{\gamma}=\dot{u}\partial_u+\dot{v}\partial_v$. Thus, the unit normal $\mathbf{n}$ is 
\begin{equation}\label{pt1}
 \mathbf{n}=\frac{1}{\Vert\dot{\gamma}\Vert}\left(-\dot{v}\cosh\frac{u}{r}\partial_u+\frac{\dot{u}}{\cosh\frac{u}{r}}\partial_v\right),
\end{equation}
where $\{\partial_u,\partial_v\}$ is the coordinate basis determined by the parametrization $\psi$. The expression of the gradient of $f$ with respect to this basis is
\begin{equation}\label{pt2}
 \nabla f= f_u\partial_u +\frac{f_v}{\cosh^2\frac{u}{r}} \partial_v.
\end{equation}
It follows from Eqs. \eqref{pt1} and \eqref{pt2} that
$$
 \langle\mathbf{n},\nabla f\rangle=\frac{1}{\Vert\dot{\gamma}\Vert}\left(-\dot{v}\cosh\frac{u}{r}f_u+\frac{\dot{u}}{\cosh\frac{u}{r}}f_v\right).
$$
This equation, together with Eq. \eqref{k-c}, finally proves Eq. \eqref{et1}.
\end{proof} 


Finally, we are in a position to provide the following characterization for extrinsic catenaries.

\begin{theorem}\label{tkilling} 
A regular curve $\gamma(t)=\psi(u(t),v(t))$ in $\mathbb{H}^2(r)$ is an extrinsic catenary if, and only if, its curvature $\kappa$ in $\mathbb{H}^2(r)$ satisfies
\begin{equation}\label{et2}
 \kappa(t) = -\frac{\langle\mathbf{n}(t), X(t)\rangle}{\mathrm{dist}(\gamma(t),P^2)},
\end{equation}
where $X$ is the Killing field in $\mathbb{E}_1^4$ of the geodesics in $\mathbb{E}_1^4$ connecting $\gamma$ to a point of $P^2$.
\end{theorem}

\begin{proof}
If $X$ is a vector field in $\mathbb{E}_1^4$, its tangent part $X^\top$ in $\mathbb{H}^2(r)$ is given by 
\begin{equation}\label{top}
X^\top=X+\frac{1}{r^2}\langle X,\psi\rangle\psi,
\end{equation}
where $\psi$ is the parametrization \eqref{eq35} of $\mathbb{H}^2(r)$. Consider now the three types of  Killing vector fields that determine the distance between a point of $\mathbb{H}^2(r)$ and a subspace $P^2$ depending on the causal character of $P^2$. Denote by $\{\mathbf{e}_x,\mathbf{e}_y,\mathbf{e}_z\}$ the canonical basis of $\mathbb{E}_1^3$. 

\begin{enumerate}
 \item If $P^2=[\bfe_{x},\bfe_{y}]$, the Killing vector field is  $X=\mathbf{e}_z$ and, from \eqref{top}, we have
 $$
  X^\top= \frac{1}{2} \sinh \frac{2 u}{r} \cosh \frac{v}{r}\mathbf{e}_x+\frac{1}{2} \sinh \frac{2 u}{r} \sinh \frac{v}{r}\mathbf{e}_y+\cosh ^2\frac{u}{r}\mathbf{e}_z.
 $$
 \item If $P^2=[\mathbf{e}_y,\mathbf{e}_z]$, the Killing vector field is $Y=-\mathbf{e}_x$ and, from \eqref{top}, we have
 $$
  Y^\top=(\cosh ^2\frac{u}{r} \cosh ^2\frac{v}{r}-1)\mathbf{e}_x+\frac{1}{2} \cosh ^2\frac{u}{r} \sinh \frac{2 v}{r}\mathbf{e}_y+\frac{1}{2} \sinh \frac{2 u}{r} \cosh \frac{v}{r}\mathbf{e}_z.
 $$
 \item If $P^2= [\mathbf{e}_{y}+\mathbf{e}_{x},\mathbf{e}_z]$, the Killing vector field $Z=-\frac{1}{r\sqrt{2}}(\mathbf{e}_x+\mathbf{e}_y)$ and, from \eqref{top}, we have
 \begin{eqnarray*}
  Z^\top & = & \frac{1}{r\sqrt{2}}\left(-1+\rme^{-\frac{v}{r}} \cosh ^2\frac{u}{r} \cosh \frac{v}{r}\right)\mathbf{e}_x \\
  & & +\frac{1}{r\sqrt{2}}\left[-1+\rme^{-\frac{ v}{r}}\cosh^2\frac{u}{r}\sinh\frac{v}{r}\right] \mathbf{e}_y+\frac{\rme^{-\frac{v}{r}}}{2r\sqrt{2}}  \sinh \frac{2 u}{r}\mathbf{e}_z.
 \end{eqnarray*}
\end{enumerate}
Let us compute the gradient $\nabla f$ in $\mathbb{H}^2(r)$ of the distances to each of the above planes $P^2$.
\begin{enumerate}
\item In the elliptic case, $f(u,v)=r \sinh\frac{u}{r}$. Using Eq. \eqref{pt2} and $\partial_u=\psi_u$ and $\partial_v=\psi_v$ in \eqref{puv}, we obtain that $\nabla f=X^\top$.
\item In the hyperbolic case, $f(u,v)=r\cosh\frac{u}{r}\cosh\frac{v}{r}$. Using Eq. \eqref{pt2} and $\partial_u=\psi_u$ and $\partial_v=\psi_v$ in \eqref{puv}, we obtain that $\nabla f=Y^\top$.
\item In the parabolic case, $f(u,v)=\frac{1}{\sqrt{2}}e^{-v/r}\cosh\frac{u}{r}$. Using Eq. \eqref{pt2} and $\partial_u=\psi_u$ and $\partial_v=\psi_v$ in \eqref{puv}, we obtain that $\nabla f= Z^\top$.
\end{enumerate}

Finally, applying Theorem \ref{t1} gives the desired result. Indeed, 
\[
 \kappa = -\frac{\langle\mathbf{n},\nabla f\rangle}{f} = -\frac{\langle\mathbf{n},X^{\top}\rangle}{\mathrm{dist}(\gamma,P^2)} = -\frac{\langle\mathbf{n},X\rangle}{\mathrm{dist}(\gamma,P^2)}.
\]
\end{proof}

\

\section{Surfaces of revolution in hyperbolic space}
\label{sec3}

This section describes the three types of surfaces of revolution in $\mathbb{H}^3(r)$. To define them, we first exploit the fact that all isometries of $\mathbb{H}^{3}(r)$ are induced by orthogonal transformations of $\mathbb{E}_1^{4}$. Later, after introducing a convenient coordinate system for each surface type, we compute their mean curvature.

Let $P^k$ denote a $k$-dimensional subspace of $\mathbb{E}_1^{4}$ and $O(P^k)$ be the set of orthogonal transformations of $\mathbb{E}_1^{4}$ with positive determinant and that leave $P^k$ pointwise fixed. Following do Carmo-Dajczer \cite{carmodajczerrotation}, we have the following classification for surfaces of revolution in $\mathbb{H}^3(r)$.

\begin{definition}
Consider two subspaces $P^2,P^3\subset\mathbb{E}_1^{4}$ such that $P^2\subset P^3$ and $P^3\cap\mathbb{H}^{3}(r)\not=\emptyset$. Let $\gamma:I\to \bbHr= P^3\cap\mathbb{H}^{3}(r)$ be a regular curve that does not intersect $P^2$. The orbits of $\gamma$ under the action of $O(P^2)$ is called a \emph{surface of revolution} in $\mathbb{H}^{3}(r)$ generated by $\gamma$ and rotated around $P^2$. In addition,  the surface of revolution is said to be of
\begin{enumerate}
 \item \emph{elliptic type}, if $P^2$ is Lorentzian,
 \item \emph{hyperbolic type}, if $P^2$ is Riemannian, and
 \item \emph{parabolic type}, if $P^2$ is degenerate. 
\end{enumerate}
\end{definition}

\begin{remark}
In the hyperboloid model, the orbits of the action of a revolution of elliptic, hyperbolic, and parabolic types are ellipses, hyperbolas, and parabolas, respectively. In \cite{carmodajczerrotation}, surfaces of elliptic type are said to be spherical. However, the orbits are elliptic unless $P^2$ is orthogonal to the $x$-axis (i.e., the canonical timelike direction). Finally, note that from an intrinsic viewpoint, we can alternatively describe the orbits as circles, hypercycles, and horocycles.
\end{remark}

In the constructions below, we first fix a basis $\{\mathbf{e}_1,\mathbf{e}_2,\mathbf{e}_3,\mathbf{e}_4\}$ of $\mathbb{E}_1^4$ (not necessarily the canonical one). Later, we shall take $P^2=[\bfe_1,\bfe_2]$ and $P^3=[\bfe_1,\bfe_2,\bfe_3]$. Finally, to generate a surface of revolution, we take a curve $\gamma(t)=x_1(t)\bfe_1+x_2(t)\bfe_2+x_3(t)\bfe_3$ and then rotate it around $P^2$. Without loss of generality, the parametrizations of the three types of surfaces of revolution are the following:
\begin{enumerate}
\item Surfaces of elliptic type. Consider $P^2=[\bfe_{x},\bfe_{y}]$ and $P^3=[\bfe_{x},\bfe_{y},\bfe_{z}]$. The members $ \mathcal{E}_{\theta}$ of $O(P^2)$ take the form
\begin{equation*}
 \mathcal{E}_{\theta} =
 \left(
 \begin{array}{cccc}
 1\quad & 0 \,\,& 0 & 0 \\
 0 \quad & 1 \,\,& 0 & 0 \\
 0 \quad & 0 \,\,& \cos\theta & -\sin\theta \\
 0 \quad & 0 \,\,& \sin\theta & \cos\theta\\
 \end{array}
 \right).
\end{equation*}
Given a curve $\gamma(t)=(x(t),y(t),z(t),0)\in \bbHr=P^3\subset\mathbb{H}^3(r)$, we obtain a surface of revolution $S_{\mathcal{E}}$ parametrized by
\begin{equation}\label{eq::ParametrizEllipticSurfOfRevH3}
 S_{\mathcal{E}}(t,\theta) = S_{\mathcal{E}}(t,\theta)=\mathcal{E}_{\theta}(\gamma(t))=(x(t),y(t),z(t)\cos\theta,z(t)\sin\theta).
\end{equation}

\item Surfaces of hyperbolic type. Consider $P^2=[\mathbf{e}_y,\mathbf{e}_z]$ and $P^3=[\mathbf{e}_y,\mathbf{e}_{z},\mathbf{e}_{x}]$. The members $ \mathcal{H}_{\theta}$ of $O(P^2)$ take the form
\begin{equation*}
 \mathcal{H}_{\theta} =
 \left(
 \begin{array}{cccc}
 \cosh\theta \quad & 0 \quad & 0 \quad & \sinh\theta \\
 0 \quad & 1 \quad & 0 \quad & 0 \\
 0 \quad & 0 \quad & 1 \quad & 0 \\
 \sinh\theta \quad & 0 \quad & 0 \quad &\cosh\theta \\
 \end{array}
 \right).
\end{equation*}
Given a curve $\gamma(t)=(x(t),y(t),z(t),0)\in \bbHr=P^3\subset\mathbb{H}^3(r)$, we obtain a surface of revolution $S_{\mathcal{H}}$ parametrized by
\begin{equation}\label{eq::ParametrizHyperbolicSurfOfRevH3}
 S_{\mathcal{H}}(t,\theta) = \mathcal{H}_{\theta}(\gamma(t))=(x(t)\cosh\theta,y(t),z(t),x(t)\sinh\theta).
\end{equation}

\item Surfaces of parabolic type. Consider $P^2=[\frac{\mathbf{e}_{y}+\mathbf{e}_{x}}{\sqrt{2}},\mathbf{e}_z]$ and $P^3=[\frac{\mathbf{e}_{y}+\mathbf{e}_{x}}{\sqrt{2}},\mathbf{e}_z,\frac{\mathbf{e}_{y}-\mathbf{e}_{x}}{\sqrt{2}}]$. The members $ \mathcal{P}_{\theta}$ of $O(P^2)$ take the form
\begin{equation*}
 \mathcal{P}_{\theta} =
 \left(
 \begin{array}{cccc}
 1+\frac{\theta^2}{2} & -\frac{\theta^2}{2} \quad & 0 \quad & -\theta \\[4pt]
 \frac{\theta^2}{2} & 1-\frac{\theta^2}{2} \quad & 0 \quad & -\theta \\[4pt]
 0 & 0 \quad & 1 \quad & 0 \\
 -\theta & \theta \quad & 0 \quad & 1 \\
 \end{array}
 \right).
\end{equation*}
Now, given a curve $\gamma(t)=(x(t),y(t),z(t),0)\in \bbHr=P^3\subset\mathbb{H}^3(r)$, we obtain a surface of revolution $S_{\mathcal{P}}(t,\theta)=\mathcal{P}_{\theta}(\gamma(t))$ parametrized by
\begin{equation}\label{eq::ParametrizParabolicSurfOfRevH3}
 S_{\mathcal{P}}(t,\theta) = \left(x+\frac{\theta^2}{2}(x-y),y+\frac{\theta^2}{2}(x-y),z,-\theta (x- y)\right).
\end{equation}
\end{enumerate}

Let us calculate the mean curvature of the surfaces of revolution parametrized as above. 

\begin{proposition}\label{prop::MeanCurvOfSurfRevolutionInH3}
Let $S_{\gamma}$ be a surface of revolution in $\mathbb{H}^3(r)\subset\mathbb{E}_1^4$ with generating curve $\gamma(t)=(x(t),y(t),z(t),0)$. Let $H$ be the mean curvature of $S_{\gamma}$. We have,
\begin{enumerate}
 \item If $S_{\gamma}$ is of spherical type, then
 \begin{equation}\label{mean-e}
 H = \frac{z\left[x(\ddot{y}\dot{z}-\dot{y}\ddot{z})-y(\ddot{x}\dot{z}-\dot{x}\ddot{z})+z(\ddot{x}\dot{y}-\dot{x}\ddot{y})\right]+(x\dot{y}-\dot{x}y)\Vert\dot{\gamma}\Vert^2}{2rz\Vert\dot{\gamma}\Vert^{3}}.
 \end{equation}
 \item If $S_{\gamma}$ is of hyperbolic type, then
 \begin{equation}\label{mean-h}
 H = \frac{x\left[x(\ddot{y}\dot{z}-\dot{y}\ddot{z})-y(\ddot{x}\dot{z}-\dot{x}\ddot{z})+z(\ddot{x}\dot{y}-\dot{x}\ddot{y})\right]-(y\dot{z}-\dot{y}z)\Vert\dot{\gamma}\Vert^2}{2rx\Vert\dot{\gamma}\Vert^{3}}.
 \end{equation}
 \item If $S_{\gamma}$ is of parabolic type, then
 \begin{eqnarray}\label{mean-p}
 H & = & -\frac{(x-y)\left[x(\ddot{y}\dot{z}-\dot{y}\ddot{z})-y(\ddot{x}\dot{z}-\dot{x}\ddot{z})+z(\ddot{x}\dot{y}-\dot{x}\ddot{y})\right]}{2r(x-y)\Vert\dot{\gamma}\Vert^{3}} \nonumber\\
 &  & - \frac{\left[(x-y)\dot{z}-(\dot{x}-\dot{y})z\right]\Vert\dot{\gamma}\Vert^2}{2r(x-y)\Vert\dot{\gamma}\Vert^{3}}.
 \end{eqnarray}
\end{enumerate}
\end{proposition}
\begin{proof}
First, note we may employ the ternary product $\times_1$ of $\mathbb{E}_1^4$ to define a vector product $\times$ in $\mathbb{H}^3(r)$. Indeed, if $X,Y\in T_p\mathbb{H}^3(r)$, we have
\begin{equation*}
X\times Y = X \times_1 Y \times_1 \frac{p}{r} 
 = \frac{1}{r}\det\left(
 \begin{array}{cccc}
 x_1 & y_1 & z_1 & w_1 \\
 x_2 & y_2 & z_2 & w_2 \\
 x & y & z & w \\
 -\bfe_x & \bfe_y & \bfe_z & \bfe_w \\
 \end{array}
 \right),
\end{equation*}
where $p=(x,y,z,w)$, $X=(x_1,y_1,z_1,w_1)$, and $Y=(x_2,y_2,z_2,w_2)$. Here, the vectors of the orthonormal basis $\{\bfe_x, \bfe_y, \bfe_z, \bfe_w\}$ are placed in the last line to guarantee that $\bfe_x \times_1 \bfe_y \times_1 \bfe_z = \bfe_w$.

Consider a parametrization $\Phi=\Phi(u^1,u^2)$ of a surface of $\mathbb{H}^3(r)$. Recall that the coefficients $(g_{ij})$ of the first fundamental form are $g_{ij}=\langle\frac{\partial\Phi}{\partial u^i}, \frac{\partial\Phi}{\partial u^j}\rangle$. Fix the unit normal 
 $$
  \xi=\frac{1}{\sqrt{\det g_{ij}}}\frac{\partial \Phi}{\partial u^1}\times\frac{\partial \Phi}{\partial u^2}.
 $$
The coefficients $(h_{ij})$ of the second fundamental form are
 $$
  h_{ij}=\langle\nabla_{\frac{\partial \Phi}{\partial u^i}}\frac{\partial \Phi}{\partial u^j},\xi\rangle = \langle\frac{\partial^2 \Phi}{\partial u^i\partial u^j},\xi\rangle.
 $$
Then, the mean curvature $H$ of $\Phi(u^1,u^2)$ is given by 
\begin{equation}\label{hh}
H=\frac{g_{22}h_{11}-2g_{12}h_{12}+g_{11}h_{22}}{2(g_{11}g_{22}-g_{12}^2)}.
\end{equation}
We now compute \eqref{hh} in each of the three types of surfaces of revolution.

\begin{enumerate}
\item Let $S_{\gamma}$ be of revolution of spherical type. Without loss of generality, we may parametrize it as in \eqref{eq::ParametrizEllipticSurfOfRevH3}: $S_{\gamma}=S_{\mathcal{E}}$. The tangent plane is spanned by 
\begin{equation*}
\frac{\partial S_{\mathcal{E}}}{\partial t} = (\dot{x},\dot{y},\dot{z}\cos\theta,\dot{z}\sin\theta)
\quad \mbox{and} \quad \frac{\partial S_{\mathcal{E}}}{\partial \theta} = (0,0,-z\sin\theta,z\cos\theta).
\end{equation*}
Thus, the first fundamental form of $S_{\mathcal{E}}$ becomes
\begin{equation*}
\rmd s^2 = (-\dot{x}^2+\dot{y}^2+\dot{z}^2)\rmd t^2+z^2\rmd\theta^2 = \Vert\dot{\gamma}\Vert^2\rmd t^2+z^2\rmd\theta^2.
\end{equation*}

The unit normal $\xi$ of $S_{\mathcal{E}}$ is given by
\begin{equation*}
 \xi = \frac{1}{r\Vert\dot{\gamma}\Vert}\Big(y\dot{z}-\dot{y}z,x\dot{z}-\dot{x}z,(\dot{x}y-x\dot{y})\cos\theta,(\dot{x}y-x\dot{y})\sin\theta \Big).
\end{equation*}
Let us compute the second fundamental form's coefficients $h_{ij}$. The second derivatives of $S_{\mathcal{E}}$ are
\[
\frac{\partial^2 S_{\mathcal{E}}}{\partial t^2} = (\ddot{x},\ddot{y},\ddot{z}\cos\theta,\ddot{z}\sin\theta), \quad  
\frac{\partial^2 S_{\mathcal{E}}}{\partial t\partial\theta} = (0,0,-\dot{z}\sin\theta,\dot{z}\cos\theta),
\]
and
\[
 \frac{\partial^2 S_{\mathcal{E}}}{\partial \theta^2} = (0,0,-z\cos\theta,-z\sin\theta).
\]
Thus, we have $h_{12}=0$ and 
\[
 h_{11} = \frac{1}{r\Vert\dot{\gamma}\Vert}\left[x(\ddot{y}\dot{z}-\dot{y}\ddot{z})-y(\ddot{x}\dot{z}-\dot{x}\ddot{z})+z(\ddot{x}\dot{y}-\dot{x}\ddot{y})\right],\quad
 h_{22} = \frac{z(x\dot{y}-\dot{x}y)}{r\Vert\dot{\gamma}\Vert}.
\]
Substitution of $h_{ij}$ and $g_{ij}$ in Eq. \eqref{hh} gives the desired expression \eqref{mean-e}. 

\item Let $S_{\gamma}$ be a surface of revolution of hyperbolic type. Without loss of generality, we may parametrize it as in \eqref{eq::ParametrizHyperbolicSurfOfRevH3}: $S_{\gamma}=S_{\mathcal{H}}$. The tangent plane is spanned by 
\begin{equation}
\frac{\partial S_{\mathcal{H}}}{\partial t} = (\dot{x}\cosh\theta,\dot{y},\dot{z},\dot{x}\sinh\theta)
\quad \mbox{and} \quad \frac{\partial S_{\mathcal{H}}}{\partial \theta} = (x\sinh\theta,0,0,x\cosh\theta).
\end{equation}
Thus, the first fundamental form of $S_{\mathcal{H}}$ becomes
\begin{equation}
\rmd s^2 = (-\dot{x}^2+\dot{y}^2+\dot{z}^2)\rmd t^2+z^2\rmd\theta^2 = \Vert\dot{\gamma}\Vert^2\rmd t^2+x^2\rmd\theta^2.
\end{equation}

The unit normal $\xi$ of $S_{\mathcal{H}}$ is given by
\begin{equation*}
 \xi = \frac{1}{r\Vert\dot{\gamma}\Vert}\Big((y\dot{z}-\dot{y}z)\cosh\theta,x\dot{z}-\dot{x}z,\dot{x}y-x\dot{y},(y\dot{z}-\dot{y}z)\sinh\theta\Big).
\end{equation*}
Let us compute the second fundamental form's coefficients $h_{ij}$. The second derivatives of $S_{\mathcal{H}}$ are
\[
\frac{\partial^2 S_{\mathcal{H}}}{\partial t^2}  = (\ddot{x}\cosh\theta,\ddot{y},\ddot{z},\ddot{x}\sinh\theta), \quad \frac{\partial^2 S_{\mathcal{H}}}{\partial t\partial\theta} = (\dot{x}\sinh\theta,0,0,\dot{x}\cosh\theta),
\]
and
\[
\frac{\partial^2 S_{\mathcal{H}}}{\partial \theta^2} = (x\cosh\theta,0,0,x\sinh\theta).
\]
Thus, we have $h_{12}=0$ and
\[
 h_{11} = \frac{1}{r\Vert\dot{\gamma}\Vert}\left[x(\ddot{y}\dot{z}-\dot{y}\ddot{z})-y(\ddot{x}\dot{z}-\dot{x}\ddot{z})+z(\ddot{x}\dot{y}-\dot{x}\ddot{y})\right], \quad
 h_{22} = -\frac{x(y\dot{z}-\dot{y}z)}{r\Vert\dot{\gamma}\Vert}.
\]
Again, a substitution of $h_{ij}$ and $g_{ij}$ in Eq. \eqref{hh} gives \eqref{mean-h}. 

\item Let $S_{\gamma}$ be a surface of revolution of parabolic type. Without loss of generality, we may parametrize it as in \eqref{eq::ParametrizParabolicSurfOfRevH3}: $S_{\gamma}=S_{\mathcal{P}}$. The tangent plane is now spanned by 
\begin{equation*}
\frac{\partial S_{\mathcal{P}}}{\partial t} = \left((1+\frac{\theta^2}{2})\dot{x}-\frac{\theta^2}{2}\dot{y},\frac{\theta^2}{2}\dot{x}+(1-\frac{\theta^2}{2})\dot{y},\dot{z},\theta\dot{y}-\theta\dot{x}\right)
\end{equation*}
and
\begin{equation*}
 \frac{\partial S_{\mathcal{P}}}{\partial \theta} = \Big(\theta (x- {y}),\theta (x- {y}),0,y-x \Big).
\end{equation*}
Thus, the first fundamental form of $S_{\mathcal{P}}$ becomes
\begin{equation*}
\rmd s^2 = (-\dot{x}^2+\dot{y}^2+\dot{z}^2)\rmd t^2+(x-y)^2\rmd\theta^2 = \Vert\dot{\gamma}\Vert^2\rmd t^2+(x-y)^2\rmd\theta^2.
\end{equation*}

The unit normal $\xi$ of $S_{\mathcal{P}}$ is given by
\begin{eqnarray*}
 \xi & = & \left[ \frac{\theta^2}{2}(x\dot{z}-\dot{x}z)-(1+\frac{\theta^2}{2})(y\dot{z}-\dot{y}z)\right]\frac{\bfe_x}{r\Vert\dot{\gamma}\Vert} \nonumber \\
 &  & - \left[\frac{\theta^2}{2}(y\dot{z}-\dot{y}z)-(1-\frac{\theta^2}{2})(x\dot{z}-\dot{x}z) \right]\frac{\bfe_y}{r\Vert\dot{\gamma}\Vert} \nonumber \\
 &  & + \Big(x\dot{y}-\dot{x}y\Big)\frac{\bfe_z}{r\Vert\dot{\gamma}\Vert}- \theta\left[(x-y)\dot{z}-(\dot{x}-\dot{y})z\right]\frac{\bfe_w}{r\Vert\dot{\gamma}\Vert}. 
\end{eqnarray*}
Let us compute the second fundamental form's coefficients $h_{ij}$. The second derivatives of $S_{\mathcal{P}}$ are
\[
\frac{\partial^2 S_{\mathcal{P}}}{\partial t^2} = \left((1+\frac{\theta^2}{2})\ddot{x}-\frac{\theta^2}{2}\ddot{y},\frac{\theta^2}{2}\ddot{x}+(1-\frac{\theta^2}{2})\ddot{y},\ddot{z},\theta\ddot{y}-\theta\ddot{x}\right),
\]
\[
\frac{\partial^2 S_{\mathcal{P}}}{\partial t\partial\theta} = (\dot{x}-\dot{y})\Big(\theta ,\theta ,0,-1 \Big), \quad \mbox{and} \quad
\frac{\partial^2 S_{\mathcal{P}}}{\partial \theta^2} = \Big(x- {y},x- {y},0,0 \Big).
\]
Thus, we have $h_{12}=0$ and
\begin{eqnarray*}
 h_{11} & = & -\frac{1}{r\Vert\dot{\gamma}\Vert}\left[x(\ddot{y}\dot{z}-\dot{y}\ddot{z})-y(\ddot{x}\dot{z}-\dot{x}\ddot{z})+z(\ddot{x}\dot{y}-\dot{x}\ddot{y})\right],\\
 h_{22} &=& -(x-y)\frac{(x-y)\dot{z}-(\dot{x}-\dot{y})z}{r\Vert\dot{\gamma}\Vert}.
\end{eqnarray*}
Substitution of $h_{ij}$ and $g_{ij}$ in Eq. \eqref{hh} gives \eqref{mean-p}. 
\end{enumerate}
\end{proof}

\section{Minimal Surfaces of Revolution}
\label{sec4}

This section proves the first of our main results, Theorem \ref{tmain} below, concerning the characterization of minimal surfaces of revolution in hyperbolic space. To do that, we use Prop. \ref{prop::MeanCurvOfSurfRevolutionInH3} to express the mean curvature in terms of the curvature $\kappa$ of the generating curve, Prop. \ref{prop::MeanCurvSurfRevolutionUsingCurvatureGeneratingFunctionH3}. Then, we compare the resulting expression  with the curvatures of the extrinsic catenaries that appear in Theorem \ref{t4}. 

First, we need an expression for the curvature of a generic curve in $\bbHr$.


\begin{lemma}\label{lemma::KappaInH2}
The curvature $\kappa$ of $\gamma(t)=(x(t),y(t),z(t))$ in $\bbHr$ is given by
\begin{equation}\label{eq::CurvxyzInH2}
 \kappa = -\frac{x(\ddot{y}\dot{z}-\dot{y}\ddot{z})-y(\ddot{x}\dot{z}-\dot{x}\ddot{z})+z(\ddot{x}\dot{y}-\dot{x}\ddot{y})}{r\Vert\dot{\gamma}\Vert^3}.
\end{equation}
\end{lemma}
\begin{proof}
The unit normal of the hyperbolic plane $\bbHr$ seen as a surface of $\mathbb{E}_1^3$ is given by $N= \gamma/r$. Thus, it follows that we may express the curvature $\kappa$ of $\gamma(t)=(x(t),y(t),z(t))$ in $\bbHr$ as
\begin{equation}\label{eq::CurvatureInH2}
 \kappa = \frac{\langle N\times_1\dot{\gamma},\nabla_{\dot{\gamma}}\dot{\gamma}\rangle}{r\Vert\dot{\gamma}\Vert^3} = \frac{\langle\gamma\times_1\dot{\gamma},\ddot{\gamma}\rangle_1}{r\Vert\dot{\gamma}\Vert^3}.
\end{equation}
A straightforward computation shows that
\[
 \gamma\times_1\dot{\gamma} = (\dot{y}z-y\dot{z},\dot{x}z-x\dot{z},x\dot{y}-\dot{x}y).
\]
Therefore, 
\begin{eqnarray}
 \langle\gamma\times_1\dot{\gamma},\ddot{\gamma}\rangle_1 &=& -\ddot{x}(\dot{y}z-y\dot{z})+\ddot{y}(\dot{x}z-x\dot{z})+\ddot{z}(x\dot{y}-\dot{x}y) \nonumber\\
 & = & -x(\ddot{y}\dot{z}-\dot{y}\ddot{z})+y(\ddot{x}\dot{z}-\dot{x}\ddot{z})-z(\ddot{x}\dot{y}-\dot{x}\ddot{y}).\nonumber
\end{eqnarray}
Substitution of $\langle\gamma\times_1\dot{\gamma},\ddot{\gamma}\rangle_1$ in Eq. \eqref{eq::CurvatureInH2} proves \eqref{eq::CurvxyzInH2}.
\end{proof}
 

\begin{proposition}\label{prop::MeanCurvSurfRevolutionUsingCurvatureGeneratingFunctionH3}
Let $S_{\gamma}$ be a surface of revolution in $\mathbb{H}^3(r)\subset\mathbb{E}_1^4$ with generating curve $\gamma(t)=(x(t),y(t),z(t),0)\in\bbHr\subset\mathbb{H}^3(r)$. Then, $S_\gamma$ is minimal if, and only if, the curvature of its generating curve satisfies
\begin{enumerate}
 \item if $S_{\gamma}$ is of elliptic type, then
 \begin{equation}\label{rh1}
 \kappa= \frac{x\dot{y}-\dot{x}y}{rz\Vert\dot{\gamma}\Vert},
 \end{equation}
 \item if $S_{\gamma}$ is of hyperbolic type, then
 \begin{equation}\label{rh2}
 \kappa= - \frac{y\dot{z}-\dot{y}z}{rx\Vert\dot{\gamma}\Vert},
 \end{equation}
 \item if $S_{\gamma}$ is of parabolic type, then
 \begin{equation}\label{rh3}
 \kappa=- \frac{(x-y)\dot{z}-(\dot{x}-\dot{y})z}{2r(x-y)\Vert\dot{\gamma}\Vert}.
 \end{equation}
\end{enumerate}
\end{proposition}
\begin{proof}
It is enough to substitute the formula for $\kappa$ obtained in \eqref{eq::CurvxyzInH2} in each of the expressions for the mean curvature of surfaces of revolution given in Proposition \ref{prop::MeanCurvOfSurfRevolutionInH3}. 
\end{proof}

\begin{theorem}\label{tmain}
The generating curves of minimal surfaces of revolution in $\mathbb{H}^3(r)$ of elliptic, hyperbolic, and parabolic types are extrinsic catenaries of spherical, hyperbolic, and parabolic types, respectively.
\end{theorem}
\begin{proof}
Let $\gamma(t)=\psi(u(t),v(t))$ be the generating curve of a surface of revolution. We calculate the right-hand sides of Eqs. \eqref{rh1}, \eqref{rh2}, and \eqref{rh3}. Using the parametrization $\psi$ given in \eqref{eq35}, we can write 
$$
 \gamma(t)=r\left(\cosh\frac{u(t)}{r}\cosh\frac{v(t)}{r},\cosh\frac{u(t)}{r}\sinh\frac{v(t)}{r},\sinh\frac{u(t)}{r}\right).
$$
Then
\begin{eqnarray}
 \frac{x\dot{y}-\dot{x}y}{rz\Vert\dot{\gamma}\Vert} &=& \frac{\dot{v}\cosh^2\frac{u}{r}}{r\sinh\frac{u}{r}\Vert\dot{\gamma}\Vert},\label{eq::auxEqMeanCurvSurfRevolutionEllipticTypeUsingCurvGeneratingCurve}\\
 \frac{y\dot{z}-\dot{y}z}{rx\Vert\dot{\gamma}\Vert} &=& \frac{\dot{u}\sinh\frac{v}{r}-\dot{v}\cosh\frac{v}{r}\cosh\frac{u}{u}\sinh\frac{u}{r}}{r\cosh\frac{v}{r}\cosh\frac{u}{r}\Vert\dot{\gamma}\Vert},\label{eq::auxEqMeanCurvSurfRevolutionHyperbolicTypeUsingCurvGeneratingCurve}\\
 \frac{(x-y)\dot{z}-(\dot{x}-\dot{y})z}{r(x-y)\Vert\dot{\gamma}\Vert}& = &\frac{\dot{u}+ \dot{v}\cosh\frac{u}{r}\sinh\frac{u}{r}}{r\Vert\dot{\gamma}\Vert\cosh\frac{u}{r}}.\label{eq::auxEqMeanCurvSurfRevolutionParabolicTypeUsingCurvGeneratingCurve}
\end{eqnarray}
Finally, we conclude the proof using the expressions for the curvature $\kappa$ of an extrinsic catenary given in Theorem \ref{t4}. 
\end{proof}

\section{Hyperbolic horocatenary}
\label{sect::Horocatenary}

In this section, we consider a variation of the hanging chain problem introduced by L\'opez \cite{lopezcatenaryform} where one measures the distance to a reference geodesic by using the so-called horocycle distance \cite{Izumiya2009Horosphericalgeometry}. In other words, we shall employ $d_h(\ell,p)$ defined as the distance measured along the horocycle passing through $p$ and orthogonal to the geodesic $\ell$. The solutions to this problem will be called \emph{hyperbolic horocatenary} \cite{lopezcatenaryform}.
Using horocatenaries will allow us to provide an intrinsic characterization of the extrinsic catenaries of the elliptic type, Theorem \ref{thr::EllipticCatenaryAreHorocatenary}.

As done for hyperbolic catenaries, we first introduce a coordinate system adapted to the problem. In the hyperboloid model, horocycles correspond to the curves given as the intersection of $\bbHr$ with lightlike planes of $\mathbb{E}_1^3$ not passing through the origin \cite{spivak1975comprehensive}. (Thus, horocycles are implicitly written as Euclidean parabolas.) Alternatively, we may describe horocycles as the orbits of the one-parameter subgroup of parabolic rotations \cite{carmodajczerrotation}, i.e., rotations induced in $\bbHr$ by the orthogonal transformations of $\mathbb{E}_1^3$ that leave a lightlike plane $P^2$ pointwise fixed. As a concrete example, consider the plane $P^2=\mathrm{span}\{\mathbf{e}_x+\mathbf{e}_y,\mathbf{e}_z\}$, where $\{\mathbf{e}_x,\mathbf{e}_y,\mathbf{e}_z\}$ is the canonical basis given by the unit velocity vectors associated with the Cartesian coordinates $(x,y,z)$. The orthogonal transformations of $\mathbb{E}_1^3$ that leaves $P^2$ pointwise fixed correspond to a lightlike rotation $L_{\theta}$ with axis $(1,1,0)$ and whose matrix is given by \cite{Nesovic2016lightlikerotation}
\begin{equation}
 L_{\theta} = \left(
 \begin{array}{ccc}
 1+\frac{\theta^2}{2} & -\frac{\theta^2}{2} & \theta \\[5pt]
 \frac{\theta^2}{2} & 1-\frac{\theta^2}{2} & \theta \\[5pt]
 \theta & -\theta & 1 
 \end{array}
 \right).
\end{equation}

Now, let $\ell(v)$ be the geodesic 
\begin{equation}
 v \mapsto \ell(v) = r (\cosh\frac{v}{r},\sinh\frac{v}{r},0). 
\end{equation}
The horocycles orthogonal to $\ell$ are given by $u\mapsto L_{u/r}(\ell(v))$. We may then parametrize $\bbHr$ as $\phi(u,v) = L_{u/r}(\ell(v))$, which gives 
\begin{equation}
 \phi= r\left(\begin{array}{l}
 (1+\frac{u^2}{2r^2})\cosh\frac{v}{r}-\frac{u^2}{2r^2}\sinh\frac{v}{r}\\
 \frac{u^2}{2r^2}\cosh\frac{v}{r}+(1-\frac{u^2}{2r^2})\sinh\frac{v}{r}\\
 \frac{u}{r}(\cosh\frac{v}{r}-\sinh\frac{v}{r})
 \end{array}\right).
\end{equation}
Alternatively, using the identity $\cosh\frac{v}{r}-\sinh\frac{v}{r}=\rme^{-\frac{v}{r}}$, we may rewrite the parametrization as
\begin{equation}\label{Def::HoroGeodesicParametrization}
 \phi(u,v) = r\left(\cosh\frac{v}{r}+\frac{u^2}{2r^2}\rme^{-\frac{v}{r}},\sinh\frac{v}{r}+\frac{u^2}{2r^2}\rme^{-\frac{v}{r}},\frac{u}{r}\,\rme^{-\frac{v}{r}}\right).
\end{equation}

The tangent vectors of the parametrization $\phi$ are
\begin{equation}
 \phi_u = \rme^{-\frac{v}{r}}\left(\frac{u}{r},\frac{u}{r},1\right), \quad 
 \phi_v = \left(\sinh\frac{v}{r}-\frac{u^2}{2r^2}\rme^{-\frac{v}{r}},\cosh\frac{v}{r}-\frac{u^2}{2r^2}\rme^{-\frac{v}{r}},-\frac{u}{r}\,\rme^{-\frac{v}{r}}\right).
\end{equation}
Then, the coefficients $g_{ij}$ of the metric are
\begin{equation*}
 g_{11} = \rme^{-\frac{2v}{r}},
\end{equation*}
\begin{eqnarray*}
 g_{12} & = & \rme^{-\frac{v}{r}}\left(-\frac{u}{r}\sinh\frac{v}{r}+\frac{u^3}{2r^3}\rme^{-\frac{v}{r}}+\frac{u}{r}\cosh\frac{v}{r}-\frac{u^3}{2r^3}\rme^{-\frac{v}{r}}-\frac{u}{r}\rme^{-\frac{v}{r}}\right)= 0,
\end{eqnarray*}
and
\begin{eqnarray*}
 g_{22} &=& -\sinh^2\frac{v}{r}+\frac{u^2}{r^2}\rme^{-\frac{v}{r}}\sinh\frac{v}{r}-\frac{u^4}{4r^4}\rme^{-\frac{2v}{r}}+\cosh^2\frac{v}{r}-\frac{u^2}{r^2}\rme^{-\frac{v}{r}}\cosh\frac{v}{r} \nonumber\\
 &+& \frac{u^4}{4r^4}\rme^{-\frac{2v}{r}}+\frac{u^2}{r^2}\,\rme^{-\frac{2v}{r}} = 1.
\end{eqnarray*}
In short, the metric of $\bbHr$ in the coordinates system $\phi(u,v)$ is
\begin{equation}
 g = \rme^{-\frac{2v}{r}}\,\rmd u^2+\rmd v^2.
\end{equation}
This coordinate system is similar to the semi-geodesic coordinate system. From now on, we shall refer to $\phi$ as the \emph{horo-geodesic parametrization}. Note that the coordinate curves $v\mapsto \phi(u,v)$ are geodesics: $\phi_{vv}=\frac{1}{r}\phi\Rightarrow \nabla_{\phi_v}\phi_v=0$. On the other hand, the coordinate curves $u\mapsto \phi(u,v)$ are horocycles.

By construction, the coordinate horocycles are orthogonal to the geodesic $\ell(v)=\phi(0,v)$. Thus, the horocycle distance can be computed as
\begin{equation}
 d_h(\phi(u,v),\ell) = \int_0^u \rme^{-\frac{v}{r}}\,\rmd t = u\,\rme^{-\frac{v}{r}}.
\end{equation}
Note that $d_h$ is nothing but the height of $\phi$ with respect to the $xy$-plane. Therefore, we can provide an intrinsic characterization of extrinsic catenary of spherical type
\begin{theorem}\label{thr::EllipticCatenaryAreHorocatenary}
Every hyperbolic horocatenary in $\mathbb{H}^2(r)$ is an extrinsic catenary of the elliptic type. Consequently, the generating curves of minimal surfaces of revolution in $\mathbb{H}^3(r)$ of the elliptic type are horocatenaries.
\end{theorem}

\begin{remark}
The horo-geodesic coordinates are analogous to the semi-geodesic coordinates after we exchange $u$ and $v$. Thus, the Christoffel symbols are
\[
\Gamma_{11}^1 = \frac{F_u}{F}=0,\,\Gamma_{11}^2 = -FF_v=\frac{1}{r}\rme^{-\frac{2v}{r}},\,\Gamma_{12}^1=\frac{F_v}{F}=-\frac{1}{r},\,\Gamma_{12}^2=0,\,\Gamma_{22}^i=0,
\]
where $F=\rme^{-\frac{v}{r}}$. The curvature of $\gamma(t)=\phi(u(t),v(t))$ in $\mathbb{H}^2(r)$ is given by
\begin{equation}
 \kappa =\frac{\ddot{u}\dot{v}-\dot{u}(\ddot{v}+\frac{\dot{u}^2}{r}\rme^{-\frac{2v}{r}}+\frac{2\dot{v}^2}{r})}{\rme^{\frac{v}{r}}(\rme^{-\frac{2v}{r}}\dot{u}^2+\dot{v}^2)^{\frac{3}{2}}}.
\end{equation}
\end{remark}

\section{Concluding remarks}

We introduced the concept of extrinsic catenaries in the hyperbolic plane, providing novel insights into the variational formulation of curves in non-Euclidean ambient manifolds. Utilizing the hyperboloid model, we defined extrinsic catenaries as critical points of the gravitational potential functional calculated from the extrinsic distance to a fixed reference plane in the ambient Lorentzian space. We delved into the characterization of extrinsic catenaries in terms of their curvature and as solutions to a prescribed curvature problem involving specific vector fields, and we showed that the generating curve of any minimal surface of revolution in the hyperbolic space is an extrinsic catenary.

We note that catenaries of the elliptic type obey a conservation law. Indeed, the Lagrangian $L$ associated with catenaries of the elliptic type, see Eq. \eqref{fu1}, does not depend on the $v$-coordinate, which implies $\partial L/\partial\dot{v}$ is a first integral\footnote{A non-constant function is a first integral if it is constant along the solution curves of the problem $\min\int L(u,v,\dot{u},\dot{v})\rmd t$.}. Geometrically, this first integral is associated with a Clairaut-like relation, thus providing information about the angle $\vartheta$ between an extrinsic catenary of the elliptic type and the $v$-coordinate curves of the parametrization \eqref{eq35}: following Theorem 4.1 of Ref. \cite{dSL23}, we can prove that
\begin{equation}
    \frac{1}{2}\sinh\frac{2u}{r}\, \cos\vartheta = \mbox{constant}.
\end{equation}
In this context, the $v$-coordinates curves play the role of parallels if we see the hyperbolic plane as an invariant surface generated by rotations of the hyperbolic type. (We obtain \eqref{eq35} by applying $\mathcal{H}_{u/r}$, Eq. \eqref{eq::ParametrizHyperbolicSurfOfRevH3}, to the geodesic $\ell(v)=r(\cosh(v/r),\sinh(v/r),0)$ followed by the exchange $(x,y,z,w)\mapsto (x,z,w,y)$.)

Extrinsic catenaries of the hyperbolic and parabolic types have no obvious first integrals. Thus, we may ask whether similar conservation also holds for them or whether the absence of circular coordinates in the Lagrangian is just an artifact of a bad choice of coordinates for the hyperbolic plane.

In many aspects, horocycles behave as extrinsically flat curves in hyperbolic geometry, and replacing geodesics with them often leads to problems with good properties. In this work, we established that extrinsic catenaries of the elliptic type could be intrinsically characterized by replacing the extrinsic distance with the intrinsic length of horocycles orthogonal to a reference geodesic, in analogy with (intrinsic) catenaries \cite{lopezcatenaryform,dSL23}. The notion of extrinsic catenaries only makes sense when working with the hyperboloid model. Thus, we may also ask whether providing intrinsic characterizations for extrinsic catenaries of the hyperbolic and parabolic types is possible. Consequently, we ask whether it is possible to characterize the generating curve of any hyperbolic minimal surface of revolution without resorting to a specific model for $\mathbb{H}^3$. A similar question can also be posed concerning minimal surfaces of revolution in $\mathbb{S}^3$.

\bmhead{Acknowledgments}
Luiz da Silva acknowledges the support provided by the Mor\'a Miriam Rozen Gerber fellowship for Brazilian postdocs and the Faculty of Physics Postdoctoral Excellence Fellowship. {Rafael L\'opez  is a member of the IMAG and of the Research Group ``Problemas variacionales en geometr\'{\i}a'',  Junta de Andaluc\'{\i}a (FQM 325). This research has been partially supported by MINECO/MICINN/FEDER grant no. PID2020-117868GB-I00,  and by the ``Mar\'{\i}a de Maeztu'' Excellence Unit IMAG, reference CEX2020-001105-M, funded by MCINN/AEI/10.13039/501100011033/ CEX2020-001105-M.}





\end{document}